\documentclass[11pt]{amsproc}   % For Latex2e

\usepackage[english]{babel}
\usepackage{amssymb,amsmath,mathrsfs}   % For Latex2e
\usepackage[dvipsnames]{xcolor}

\usepackage{graphicx}
\usepackage{epsfig}
\usepackage{epstopdf}
\usepackage[colorlinks=true]{hyperref}
\newtheorem*{theoremA}{Theorem A}
\newtheorem{proposition}{Proposition}%[theorem]
\newtheorem{corollary}{Corollary}%[theorem]

\theoremstyle{definition}
\newtheorem{definition}{Definition}%[section]
\newtheorem{example}{Example}%[section]

\theoremstyle{remark}
\def\R#1{\mathbb{R}^{#1}}
\def\I{\mathrm{i}}

\def\half#1#2{\begin{matrix}\frac{#1}{#2}\end{matrix}}

\DeclareMathOperator{\im}{\mathrm{Im}}
\DeclareMathOperator{\re}{\mathrm{Re}}

\DeclareMathOperator{\cn}{\mathrm{cn}}

\begin{document}

\title{New examples of higher-dimensional
  minimal hypersurfaces}

\author{Sergienko Vladimir V.}

\author{Tkachev Vladimir G.}
\address{Department of Mathematics, Link\"oping University, 581 83 Link\"oping, Sweden}
\email{\texttt{vladimir.tkatjev@liu.se}}

% Address of record for the research reported here

\thanks {An updated version of a earlier preprint of 1999}
\thanks{}

%
%  Math Subject Classifications
%
%\subjclass{Primary 53C42, 49Q05; Secondary 53A35}

%
% Key words
%
%\keywords{Minimal surfaces}

\begin{abstract}
The main results of the paper are Proposition~3 and 4 which provide an effective way to construct minimal hypersurfaces in a Euclidean space. We demonstrate our technique by several new examples.

\textit{This note is English translation of an earlier note written by the authors (in Russian) in September 1999. The final version of the paper will be published somewhere else. The  authors are grateful to Prof. Hojoo~Lee (KIAS) for encouraging to complete and publish the present notes. }
\end{abstract}

\maketitle

\section{Introduction}

%\begin{figure}[h]
%\includegraphics[width=80mm]{Weiersstras03.png}
%\caption{A surface (\ref{eq:snsn}) for $k=4/5$ and its fundamental domain} \label{fig:snsn}
%\end{figure}

It is well known that any classical solution $u(x)=u(x_1,\ldots,x_n)$ of the following equation
\begin{equation}\label{mineq}
(1+|Du|^2)\Delta u-\sum_{i,j=1}^nu_{x_i} u_{x_j} u_{x_ix_j}=0,
\end{equation}
where $u_{x_i}=D_{x_i}u$, gives rise to a minimal (zero mean curvature) hypersurface
$$
x_{n+1}=u(x), \quad x\in \Omega\subset \R{n}.
$$

\subsection{Minimal surfaces with `harmonic level sets'}
In \cite{SergTk1998}, \cite{SergTk2000} and \cite{SergTk1999}, the authors studied and classified all zero mean curvature surfaces in the Euclidean and Minkowski spaces given implicitly by
\begin{equation}\label{Msurface}
\mathscr{M}=\{x\in \R{3}:\re h(z)=F(x_3)\},\qquad z=x_1+x_2 \sqrt{-1},
\end{equation}
where $h(z)$ is a holomorphic function of  one complex variable. These surfaces were referred to as minimal surfaces with `harmonic level sets' (\cite{SergTk1999}). More precisely, one has the following result in the Euclidean case.

\begin{theoremA}[\cite{SergTk1998}]
The surface $\mathscr{M}$ defined by (\ref{Msurface}) is minimal if and only if $h'(z)=1/g(z)$, where $g(z)$ satisfies
\begin{equation}\label{geqv}
g''(z)g(z)-g'^2(z)=c\in \R{}.
\end{equation}
In this case, the function  $F$ is found by $F''(t)+ {Y}(F(t))=0$, where $Y(t)$ is  well-defined by virtue of
$$
\frac{\re g'}{|g|^2}=-{Y}(\re h(z)).
$$
\end{theoremA}

\medskip
\textbf{Remark.} Notice that the resulting surface $M$, if non-empty, is automatically embedded because
$$
|\nabla f(x)|^2=F'^2(x_1)+\frac{1}{|g(z)|^2}>0,
$$
where $f(x)=F(x_1)-\re h(z)$.

A simple analysis reveals that the only possible solutions of (\ref{geqv}) are
\begin{itemize}
\item[(a)] $g(z)=az+b$,
\item[(b)] $g(z)=ae^{bz}$, and
\item[(c)] $g(z)=a \sin (bz+c)$, $a^2, b^2\in \R{}$ and $c\in \mathbb{C}$
\end{itemize}
which, in particular, yields:
\smallskip

%\fontsize{5}{7.2}\selectfont
\begin{table}

\begin{tabular}{p{0.24\linewidth}p{0.24\linewidth}p{0.24\linewidth}p{0.24\linewidth}}
\includegraphics[width=\linewidth]{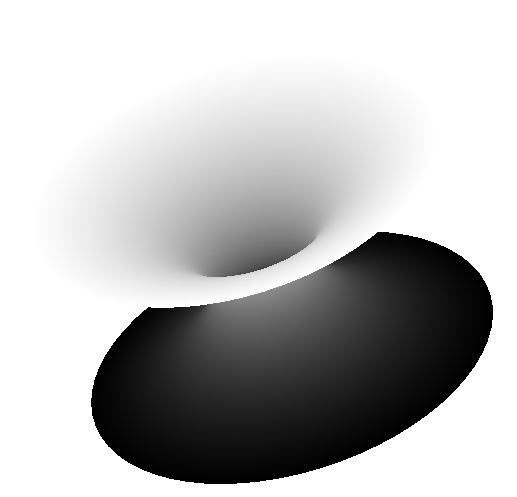}&
\includegraphics[width=\linewidth]{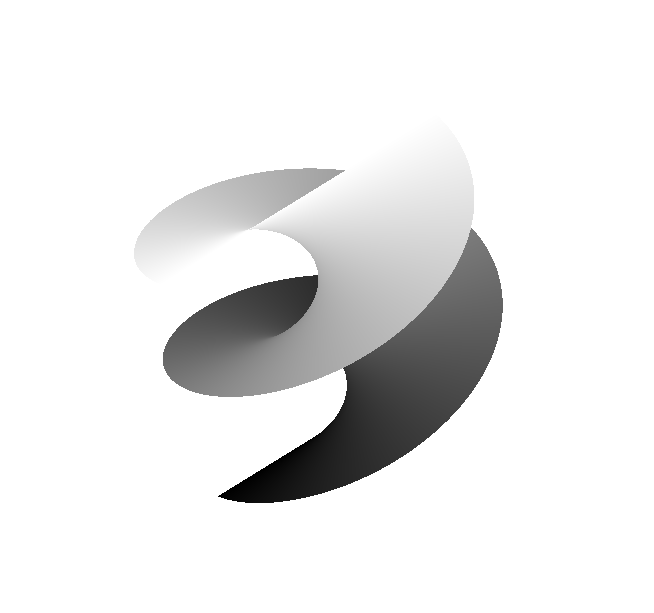}&
\includegraphics[width=\linewidth]{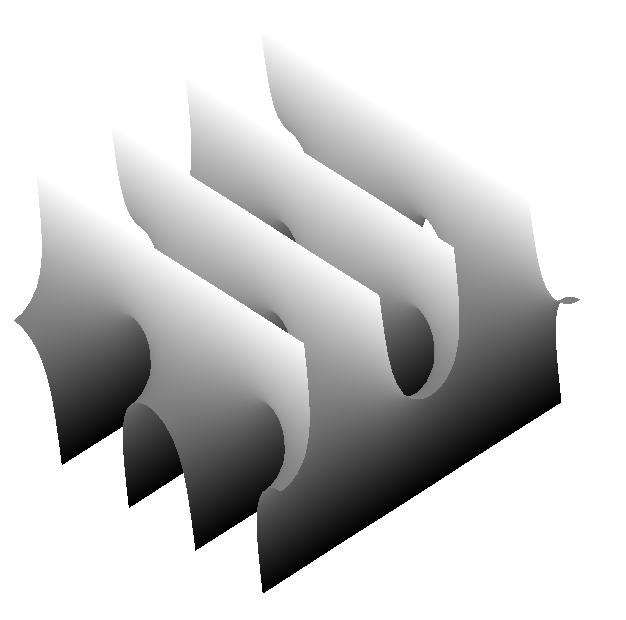}&
\includegraphics[width=\linewidth]{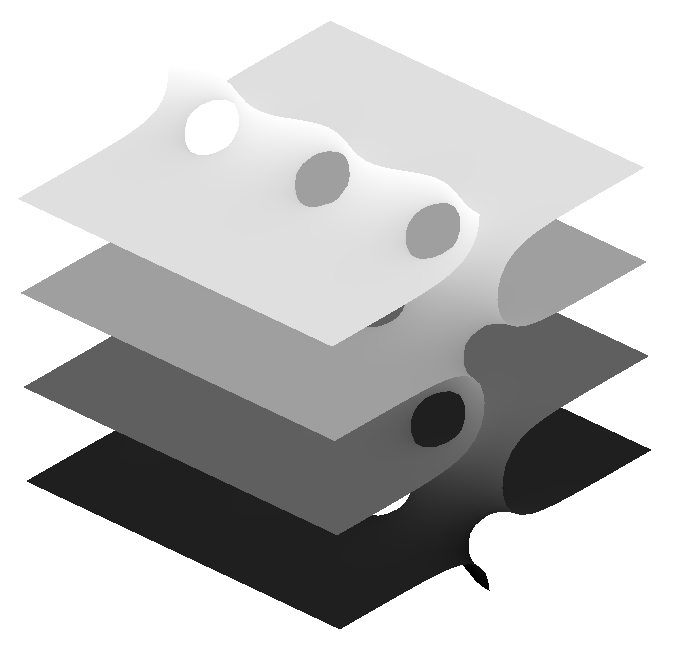}
\\
\\
the catenoid&
the helicoid &
a Scherk type surface &
a doubly periodic\\
&&&surface\\
$g(z)=z$& $g(z)=\I z$ & $g(z)=e^z$ & $g(z)=\sin z$ \\
%$F(t)=\ln \cosh t$ &$F(t)=t$ & $F(t)=\cos t$ & \\
$h(z)=\ln z$& $h(z)=-\I \ln z$ & $h(z)=-e^{-z}$ & $h(w)=-\ln \tanh \frac{z}{2}$\\
\\
%&&&\\
%\multicolumn{4}{c}{The defining equations}\\
%&&&\\
$x_1^2+x_2^2=\cosh^2 x_3$&
$\frac{x_2}{x_1}=\tan x_3$&
$\exp x_3=\frac{\cos x_2}{\cos x_1}$&
$\cn (\frac{k x_3}{k'},k)=\frac{\sin x_2}{\sinh x_1}$
\end{tabular}
\medskip
\caption{Some partial solutions}
\end{table}

\section{The higherdimensional case}
Let $f=f(\xi)$ be of class $C^2$ in an open set of $\R{N}$. It is well known (see, for instance, \cite{Hsiang67}) that the non-singular zero-locus
$$
\mathscr{M}_f:=f^{-1}(0)\cap \{\xi:|D f(\xi)|\ne0\}
$$
is a  minimal hypersurface in $\R{N}$ if and only if
\begin{equation}\label{minimalcond}
\Delta_1 f(\xi)=0 \mod f(\xi).
\end{equation}
Here
$$
\Delta_1f:=|Df|^2\Delta f-\sum_{i,j=1}^N f'_{\xi_i} f'_{\xi_j}f'' _{\xi_i\xi_j}
$$
is the mean curvature operator (in fact, the $1$-Laplace operator) and the congruence
$$
a(\xi)\equiv b(\xi) \mod c(\xi)
$$
is understood in the sense that $c(\xi)=0$ implies $a(\xi)-b(\xi)=0$. Notice that the hypersurface is automatically embedded as a level set.

Below, we study a generalization of (\ref{Msurface}), i.e. the solutions $f(\xi)$ of (\ref{minimalcond})  having the form
\begin{equation}\label{fedf}
f(\xi)=\re h(z)-F(t), \quad \xi=(z,t)\in \mathbb{C}^m\times \R{k}=\R{2m+k},
\end{equation}
where $h(z)$ is a holomorphic function and $F(t)$ is a function of class $C^2$. More precisely, recall that a function $h(z)=h(z_1,\ldots,z_k):\Omega\to \mathbb{C}$ is called holomorphic in $\Omega\subset \mathbb{C}^n$, or $h\in \mathcal{O}(\Omega)$, if $D_{\bar z_k} h(z)= 0$ in $\Omega$ for any $k=1,\ldots,m $, where
$$
D_{z_k}=\half12(D_{x_k}-\sqrt{-1}D_{y_k}),\qquad
D_{\bar z_k}=\half12(D_{x_k}+\sqrt{-1}D_{y_k})
$$
and $z_k=x_k+\sqrt{-1}y_k\in \mathbb{C}$, $k=1,\ldots,m$.
The reader easily verifies that the following fact holds true.

\begin{proposition}\label{pro:lem}
Let $h=h(z)\in \mathcal{O}(\Omega)$. Then
\begin{equation}\label{ident1}
\begin{split}
D_{x_k}\re h &=\phantom{-}D_{y_k}\im h =\re  h'_{z_k} \\
D_{x_k}\im h &=-D_{y_k}\re h =\im  h'_{z_k}
\end{split}
\end{equation}
and
\begin{equation}\label{ident2}
\begin{split}
D^2_{x_ix_j}\re h &=D^2_{x_iy_j}\im h =-D^2_{y_iy_j}\im h =\re h''_{z_iz_j}\\
D^2_{x_ix_j}\im h &=-D^2_{x_iy_j}\re h =-D^2_{y_iy_j}\im h =\im h''_{z_iz_j}\\
\end{split}
\end{equation}
\end{proposition}

It follows from (\ref{ident1}) that
$$
|D_zh|^2:=\sum_{i=1}^m |h'_{z_i}|^2=|D_{(x,y)}\re h|^2=|D_{(x,y)}\im h|^2.
$$

\begin{proposition}
\label{pro:1}
Let $f$ is defined by (\ref{fedf}) and $\mathscr{M}_f\ne \emptyset$. The  surface $\mathscr{M}_f$ is an (embedded) minimal hypersurface if and only if
\begin{equation}\label{defequa}
\re\sum_{i,j=1}^{m} \overline{h_{z_iz_j}''}h'_{z_i}h'_{z_j}\equiv -|D_z h|^2\Delta F -\Delta_1 F \mod ( F-\re h)
\end{equation}
holds whenever
\begin{equation}\label{implicit}
|D_\xi f|^2=\sum_{i=1}^m |h_{z_k}'|^2+|D_t F|^2\ne0.
\end{equation}
\end{proposition}

\begin{proof}
Under the non-singularity assumption (\ref{implicit}), it suffices to show that (\ref{minimalcond}) is equivalent to (\ref{defequa}). To this end, we notice that by Proposition~\ref{pro:lem}
$$
f'_{x_i}=\re h'_{z_i}, \quad
f'_{y_j}=-\im h'_{z_i}, \quad
f'_{t_k}=-F_{t_k}.
$$
Further, $f''_{x_it_k}=f''_{y_jt_k}=0$, $f''_{t_it_j}=-F''_{t_it_j}$ and
$$
f''_{x_ix_j}=-f''_{y_iy_j}=\re h''_{z_iz_j}, \qquad
f''_{x_iy_j}=f''_{x_jy_i}=-\im h''_{z_iz_j}.
$$
This readily yields
\begin{equation*}
\begin{split}
\sum_{i,j=1}^{2m+k}f'_{x_i}f'_{x_j}f''_{x_ix_j}&=
\re\sum_{i,j=1}^{m} \overline{h_{z_iz_j}''}h'_{z_i}h'_{z_j}- \sum_{i,j=1}^{k} F'_{t_i}F'_{t_j}F''_{t_it_j}
\end{split}
\end{equation*}
hence
\begin{equation*}
\begin{split}
-\Delta_1 f&=\bigl(\sum_{k=1}^m |h_{z_k}'|^2+|DF|^2)\Delta F+
\re\sum_{i,j=1}^{m} \overline{h_{z_iz_j}''}h'_{z_i}h'_{z_j}- \sum_{i,j=1}^{k} F'_{t_i}F'_{t_j}F''_{t_it_j}\\
&=|D_z h|^2\Delta F+\Delta_1 F+\re\sum_{i,j=1}^{m} \overline{h_{z_iz_j}''}h'_{z_i}h'_{z_j},\\
\end{split}
\end{equation*}
and the desired property follows.\end{proof}

\section{Applications}

\subsection{The class $\mathscr{T}^m$}
Here we show some examples  how Proposition~\ref{pro:1} can be used to construct minimal hypersurfaces in $\R{2m}$.

\begin{definition}
We say that a holomorphic function $h(z_1,\ldots,z_m)$ is $\R{}$-holomporphic equivalently write $h\in \mathscr{T}^m$, if there exists  a real valued function $\mu:\mathbb{C}^m\to \R{}$ such that the relation
\begin{equation}\label{muh}
\sum_{i,j=1}^{m} \overline{h_{z_iz_j}''}h'_{z_i}h'_{z_j}=\mu(z) {h(z)}
\end{equation}
holds everywhere in the domain of holomorphy  of $h$.
\end{definition}

It follows from the standard theory that if a function $h$ is $\R{}$-holomporphic  in some open set then it is $\R{}$-holomporphic everywhere in the domain of holomorphy.   The importance of the introduced class follows from the proposition below.
\begin{proposition}
\label{pro:lem2}
Let $h(z)\in \mathscr{T}^m$. Then
$$
\mathscr{M}_{\re h}=\{z\in \mathbb{C}^m=\R{2m}:\,\,\re h(z)=0, |D h(z)|\ne0\}
$$
is an embedded minimal hypersurface.
\end{proposition}

\begin{proof}
It follows from (\ref{muh}) that
$$
\re\sum_{i,j=1}^{m} \overline{h_{z_iz_j}''}h'_{z_i}h'_{z_j}=\mu(z) \re {h(z)},
$$
hence applying  Proposition~\ref{pro:1} to $h(z)$ and $F(z)\equiv 0$ yields by (\ref{defequa}) the required conclusion.
\end{proof}

Though the  case $m=1$ is trivial (to get a non-trivial minimal hypersurface one need to have at least $m=2$), it still is very useful for the further constructions. We have the following complete classification of $\mathscr{T}^1$.

\begin{proposition}
\label{pro:m}
Any element of $\mathscr{T}^1$ is either a binomial $h(z)=(az+b)^p$ or the exponential $h(z)=e^{pz}$, where  $a,b\in \mathbb{C}$ and $p\in \R{}$.
\end{proposition}

\begin{proof}
Indeed, let $\Omega$ be the domain of holomoprhy of $h(z)$. Then (\ref{muh}) yields
$$
\frac{|h''(z)|^2}{\mu(z)}=\frac{h''(z)h(z)}{h'^2(z)},
$$
where the right hand side is a meromorphic function in $\Omega$, while the left hand side is real-valued in $\Omega$. Thus, the both sides are  constant in $\Omega$, say equal to $c\in \R{}$. This yields $ch'^2(z)=h''(z)h(z)$, or $h'(z)=ch(z)^b$ for some real $b$. This  yields the required conclusions.
\end{proof}

The following proposition shows that  $\R{}$-holomporphic functions have a nice multiplicative structure.

\begin{proposition}\label{pro:4}
Let $h(z)\in \mathscr{T}^{m}$ and $g(w)\in \mathscr{T}^{n}$. Then
\begin{itemize}
\item[(i)]
$ch(z)^r\in \mathscr{T}^{n}$ for any $c\in \mathbb{C}$ and $r\in \R{}$;

\item[(ii)]
$h(z)g(w)\in \mathscr{T}^{m+n}$.

\item[(iii)]
$h(z)/g(w)\in \mathscr{T}^{m+n}$.
\end{itemize}
\end{proposition}

\begin{proof}
Setting $H(z):=h(z)^r$ one easily verifies that
$$
\sum_{i,j=1}^{m} \overline{H_{z_iz_j}''}H'_{z_i}H'_{z_j}=
r^3\sum_{i,j=1}^{m} ((r-1) |h|^{2r-4}|D_zh|^4+\mu |h|^{2r-2})h^r=\mu_1 H,
$$
with $\mu_1=r^3\sum_{i,j=1}^{m} ((r-1) |h|^{2r-4}|D_zh|^4+\mu |h|^{2r-2})$, obviously a real-valued function, thus yielding $h^r\in \mathscr{T}^{m}$. Similarly one justifies $ch\in \mathscr{T}^{m}$ which yields (i).

Further, we have
$$
\sum_{i,j=1}^{m} \overline{h_{z_iz_j}''}h'_{z_i}h'_{z_j}=\mu h, \qquad
\sum_{\alpha,\beta=1}^{n} \overline{g_{w_{\alpha}w_{\beta}}''}g'_{w_\alpha}g'_{w_\beta}=\nu g,
$$
where $\mu(z)$ and $\nu(w)$ are real-valued functions. Therefore, setting $H(z,w):=h(z)g(w)$ one obtains
\begin{equation*}
\begin{split}
\sum_{k,l=1}^{m+n} \overline{H_{z_iz_j}''}H'_{z_i}H'_{z_j}&=
|g|^2g\sum_{i,j=1}^{m} \overline{h_{z_iz_j}''}h'_{z_i}h'_{z_j}+
|h|^2h\sum_{\alpha,\beta=1}^{n} \overline{g_{w_{\alpha}w_{\beta}}''}g'_{w_\alpha}g'_{w_\beta}+2|D_zh|^2|D_w g|^2\,hg\\
&=(\mu |g|^2+\nu|h|^2+2|D_zh|^2|D_w g|^2)H,
\end{split}
\end{equation*}
yielding (ii). Finally, setting $r=-1$ and $c=1$ in (i) implies that $1/h(z)\in \mathscr{T}^n$, thus together with (ii) implies (iii).
\end{proof}

\subsection{Examples}
\begin{example}\label{ex:0}
A trivial example of a $\R{}$-holomorphic function is a linear function $h(z_1,\ldots,z_m)$, where $\mu\equiv 0$. The corresponding minimal hypersurface is a hyperplane in $\R{2m}$.
Another simple example of a $\R{}$-holomorphic function is (by Proposition~\ref{pro:m}) the function $h(z_1)=e^{z_1}$. It can be used to construct a highly non-trivial examples as Corollary~\ref{cor:1} below shows.
\end{example}

\begin{example}\label{ex:1}
A less trivial example is  and the quadratic form
$$
h(z_1,\ldots,z_m)=z_1^2+\ldots+z_m^2
$$
which satisfies (\ref{muh}) with $\mu= 8$. The corresponding minimal hypersurface is the Clifford cone
$$
\re h(z)=x_1^2+\ldots +x_m^2-y_1^2-\ldots-y_m^2=0.
$$
Observe that the non-singularity condition
$$
|D_zh|^2=4(|z_1|^2+\ldots+|z_m|^2)\ne 0
$$
holds everywhere outside the origin of $\R{2m}$.
\end{example}

\begin{example}\label{ex:2}
Another inteersting example is the cubic form
$$
h(z)=\det \left(
            \begin{array}{ccc}
              z_1 & z_2 & z_3 \\
              z_4 & z_5 & z_6 \\
              z_7 & z_8 & z_9 \\
            \end{array}
          \right)
          =
          z_{1} z_{5} z_{9}-z_{1} z_{8} z_{6}-z_{2} z_{4} z_{9}+z_{2} z_{7} z_{6}+z_{3} z_{8} z_{4}-z_{3} z_{5} z_{7}
$$
in which case $\re h(z)$ is a irreducible cubic form in $\R{18}$. It is straightforward  to verify  that
$
\mu=2\sum_{i=1}^{9}|z_i|^2
$
in (\ref{muh}). The equation $\re h(z)=0$ yields a known example of a \textit{cubic} minimal hypercone in $\R{18}$ (see also \cite{Tk10c}, \cite{Tk15}). The corresponding Hsiang algebra (REC-algebra in terminology of \cite{NTVbook}) is isomorphic to the Jordan algebra of Hermitian $3\times3$-matrices over $\mathbb{C}$ .
\end{example}

Below we briefly demonstrate how to use   Examples~\ref{ex:0}-\ref{ex:3} and Proposition~\ref{pro:4} to constructing new examples. For instance, we have the following construction of minimal hypersurfaces in \textit{odd-dimensional} ambient spaces.

\begin{corollary}
\label{cor:1}
Let $h(z)\in \mathscr{T}^m$. Then
\begin{equation}\label{arg1}
x_{2m+1}=\arg h(z)
\end{equation}
is a minimal hypersurface in $\R{2m+1}$.
\end{corollary}

\begin{proof}
Indeed, the function $g(z_1,\ldots,z_m,z_{m+1}):=ie^{z_{m+1}}h(z_1,\ldots,z_m)$ is $\R{}$-holomorphic by Proposition~\ref{pro:4} and Proposition~\ref{pro:m}. Then
$$
\re g=-e^{\re z_{m+1}}(\re h\sin\im  z_{m+1}+\im h\cos\im  z_{m+1})
$$
yields that $\re g=0$ is equivalently defined by
$$
 \im z_{m+1}=-\arctan \frac{\im h}{\re h}=-\arg h+C
$$
for some real constant $C$. It is easily seen that the latter equation is equivalent to  (\ref{arg1}) up to an orthogonal transformation (a reflection) of $\R{2m+1}$.
\end{proof}

\begin{example}
Setting $h(z_1)=z_1=x_1+i x_2$,  (\ref{arg1}) becomes $x_3=\arctan x_2/x_1$, i.e. the classical helicoid. More generally, one has the following minimal hypersurface:
$$
x_{2m+1}=\arg (z_1^{k_1}...z_m^{k_m}), \qquad k_i\in \mathbb{Z}.
$$
\end{example}

Combining Example~\ref{ex:1} and Proposition~\ref{pro:4} yields.
\begin{corollary}
\label{cor:Lawson}
Let natural numbers $p_i$, $1\le i\le m$, be subject to the GCD-condition $(p_1,\ldots,p_m)=1$ and let $c\in \mathbb{C}^\times$. Then the hypersurface $$
\re (cz_1^{p_1}\cdots z_m^{p_m})=0,
$$
 is minimal (in general singularly) immersed cone in $\R{2n}\cong \mathbb{C}^m$.
\end{corollary}

\begin{example}\label{ex:3}
For $c=1$, Corollary~\ref{cor:Lawson} yields exactly the observation made earlier by H.B.~Lawson~\cite[p.~352]{Lawson}. For instance, when $m=2$ one obtains the well-known infinite family of immersed algebraic minimal Lawson's hypercones $\re (z_1^pz_2^q)=0$, $(p,q)=1$, in $\R{4}$. The intersection of such a cone with the unit sphere $S^3$ is an immersed minimal surface of Euler characteristic zero of $S^3$ \cite{Lawson}.
\end{example}

\begin{example}\label{ex:4}
For $c=\sqrt{-1}$, Corollary~\ref{cor:Lawson} one obtains minimal hypersurfaces in $\R{2m}$ of the following kind:
$$
\sum_{i=1}^m p_i\arctan \frac{y_k}{x_k}=0
$$
which obviously is an algebraic minimal cone in $\R{2m}$.
\end{example}

\subsection{Open questions}

Concluding this short paper, we emphasize again, that many more examples can be constructed by combining Examples~\ref{ex:0}-\ref{ex:3} and Proposition~\ref{pro:4}. An interesting and important question in this direction is how to classify all $\R{}$-holomorphic functions?

Even some particular results could be interesting. For instance, all the above examples are obtained from simplest elementary blocks $h=z$ and $h=e^z$ by products (ratios) and exponentiations. Is it true that $\mathscr{T^m}$ is finitely generated?

%\begin{example}\label{ex:5}
%To provide another series of examples, let us notice that $h(z)=e^{z}\in \mathscr{T}^1$, in fact satisfying (\ref{muh}) with $\mu(z)=e^{\re z}=e^x$. Then Proposition~\ref{pro:4} readily yields that
%$$
%\re ce^{z_{m+1}}z_1^{p_1}\cdots z_m^{p_m}=0
%$$
%is a minimal hypersurface in $\R{2m+2}$. For instance, if $c=1$, $m=1$,  the latter equation becomes
%$$
%\re e^{z_{2}}z_1^{p_1}=0\quad \Leftrightarrow \quad
%\arctan\frac{y_1}{x_1}=\frac{\pi}{2p_1}-\frac{1}{p_1}y_2
%$$
%which obviously is an equation of the classical minimal surface, the helicoid! It is worth
%
%\end{example}

\def\cprime{$'$}
\providecommand{\bysame}{\leavevmode\hbox to3em{\hrulefill}\thinspace}
\providecommand{\MR}{\relax\ifhmode\unskip\space\fi MR }
\providecommand{\MRhref}[2]{%
  \href{http://www.ams.org/mathscinet-getitem?mr=#1}{#2}
}
\providecommand{\href}[2]{#2}

\end{document}